\documentclass[12pt]{article}
\usepackage{verbatim}
\usepackage{amssymb,amsmath,amsthm}
\usepackage{graphicx}
\usepackage{epsfig}
\newtheorem{theorem}{Theorem}
\newtheorem{corollary}{Corollary}[section]
\newtheorem{lemma}[corollary]{Lemma}
\newtheorem{proposition}[corollary]{Proposition}

\newcommand{\Prob} {{\mathbb P}}
 
\newcommand{\E}{{\mathbb E}}

\newcommand{\R}{{\mathbb{R}}}
\newcommand{\C}{{\mathbb C}}

\newcommand{\x}{{\bf x}}
\newcommand{\y}{{\bf y}}

\newcommand\N{{\mathbb N}}

\newcommand {{\ball}} {{\mathcal B}}

\renewcommand \Im {{\rm Im}}
\renewcommand \Re {{\rm Re}}
\newcommand \p {\partial}

\newcommand \eset {\emptyset}

\newcommand \paths {{\cal K}}
\renewcommand \loop {{\cal L}}

\newcommand \soup {{\cal C}}
\newcommand \ls {{\soup}}

\newcommand \loops {\loop}

\newcommand \cm   {{\hat m}}

\newcommand{\sequences} {{\mathcal S}}
 \newcommand {\uloop} { {\tilde \omega}}
 \newcommand {\umeas}  {\tilde m}
 \newcommand {\vf} {\phi}
 \newcommand {\trivial}{{\mathcal T}}
\newcommand {\laplace}  {{\Delta}}

\newcommand\w{\omega}

\newcommand \lmeas {{\mathcal M}}

\newcommand   \loopset {{\mathcal O}}
\begin{document}

\title{Loop measures and the Gaussian free field}

\author{Gregory F. Lawler \thanks{Research supported by National
Science Foundation grant DMS-0907143.}\\
University of Chicago  \\ \\ \\
Jacob Perlman\\ University of Chicago}

\maketitle

\begin{abstract}
Loop measures and their associated loop soups are generally viewed as arising from finite state Markov chains. We generalize several results to loop measures arising from potentially complex edge weights.   We discuss two applications:  Wilson's algorithm to produce uniform spanning trees and 
 an  isomorphism theorem due to Le Jan.  

\end{abstract}

\section{Introduction}
Loop measures have become important in the analysis of random walks and fields arising from random walks.
Such measures appear in work of Symanzik \cite{Sym}
but the
recent revival came from the Brownian loop soup \cite{LWerner}
which arose in the study of the Schramm-Loewner evolution.  The
random walk loop soup is a discrete analogue for which one
can show convergence to the Brownian loop soup.  The study of such measures and soups has continued: in continuous time by Le Jan \cite{LeJan}
 and in discrete time in \cite{LJose,
 LLimic}.   The purpose of this note is to
give an introduction to the discrete time measures and to 
discuss two of the applications: 
the relation with loop-erased walk and spanning
trees,   and a distributional identity between a
function of the loop soup and the square of
the Gaussian free field.  
 This paper is not intended to be a survey but only
 a sample of the uses of the loop measure.

While the term ``loop measure'' may seem vague, we are
talking about a specific measure from which a probabilistic
construction, the ``loop soup'' is derived.
We are emphasizing the
 loop measure  rather than the loop soup which is a Poissonian realization of the measure because we want to allow
 the loop measure to take negative or complex values.
  However, we do consider the loop soup as a complex
  measure.

We will start with some basic definitions.  In many ways, the loop measure can be considered a way to understand matrices, especially determinants,  and some of the results have very
classical counterparts.  Most of the theorems about the basic properties can be found in \cite[chapter 9]{LLimic}
  although that book restricted itself to positive measures.  We redo some proofs just to show that positivity of the entries is not important.  A key fact is
that   the total mass of the  loop measure is
the negative of the logarithm of the  determinant
of the    Laplacian. 

We next introduce the loop-erased random walk and show how one can use loop measures to give a short proof of Kirchhoff matrix-tree theorem by using an algorithm due to David
Wilson for generating uniform spanning trees. 
 
Our next section describes an isomorphism theorem found by  Le Jan that is related to earlier isomorphism theorems of Brydges, et. al. and Dynkin.  In this case, one shows that the local time of a continuous time version of the loop soup has the same
distribution as the square of a Gaussian field.  Le Jan established this by constructing a continuous-time loop soup.  We choose a slightly different, but essentially equivalent, method of using
the discrete loop soup and then adding exponential waiting times.  This is similar to the construction of continuous time Markov chains by starting with a discrete time chain and then adding the waiting times.  In order to get the formulas to work, one needs
to consider a correction term that is given by ``trivial loops''.

We finally give some discussion of complex Gaussian fields
with positive definite Hermitian weights.    We first consider real (signed) weights and relate this to the real Gaussian free field.  Finally we consider a complex Gaussian field and show that it can be considered as a pair of real Gaussian fields.

\section{Definitions}

We will consider edge weights, perhaps complex
valued,  on a finite state space
$A$. A set of weights
is the same thing as a matrix  $Q$ indexed by $A$.

\begin{itemize}

\item We call $Q$ {\em acceptable} if the matrix with entries $|Q(x,y)|$ has all eigenvalues in the interior of the unit disc.  (This is not a standard term, but we will use it for convenience.)

\item  We say $Q$ is {\em positive} if the entries are nonnegative
and $Q$ is {\em real} if the entries are real.

\item  As usual, we say that $Q$ is {\em symmetric} if $Q(x,y)
 = Q(y,x)$ for all $x,y$ and $Q$ is {\em Hermitian} if $Q(x,y)
  = \overline{Q(y,x)}$ for all $x,y$.

\item If $Q$ is Hermitian we say that $Q$ is {\em  
positive
definite} if all the eigenvalues are strictly greater
than zero, or equivalently if $\overline \x \cdot Q
\x >0$ for all non-zero $\x$.

\end{itemize}

If $A \subsetneq A'$ and $Q$ is the transition matrix
for an irreducible Markov chain on $A'$, then $Q$ restricted
to $A$ is positive and acceptable.  This is one of the main
examples of interest. If $Q$ is any   matrix, then $\lambda Q$ is
 acceptable for $\lambda > 0$ sufficiently small.

  If 
$V \subset A$ with $k$ elements, we will write $Q_V$ for the 
$k \times k$ matrix obtained by restricting $Q$
to $V$.  A {\em path} in $A$ of length $n$ is a
finite sequence of points 
\[  \omega = [\omega_0,\ldots,\omega_n] , \;\;\;\;
 \omega_j \in A.\]
We write $|\omega| = n$ for the number of steps
in the path and $\omega^R$ for the reversed path
\[   \omega^R = [\omega_n,\ldots,\omega_0] . \]
We allow the trivial paths with $|\omega| = 0$.
We write $\paths_{x,y}(A)$ for the set of
 all paths in $A$ with $\omega_0 = x, \omega_n = y$; if $x = y$,
 we include the trivial path.  

The matrix $Q$ gives  the {\em path measure} defined
  by
\[   Q(\omega) = \prod_{j=1}^n Q(\omega_{j-1},
\omega_j), \;\;\;\;\omega  = [\omega_0,\ldots,\omega_n]\in \bigcup_{x,y \in A}
  \paths_{x,y}(A) , \]
  where $Q(\omega) = 1$ if $|\omega| = 0$. 
 Note that if $Q$ is Hermitian, then
  $Q(\omega^R) = \overline {Q(\omega)}$. 
 A path $\omega$ is a {\em (rooted) loop (rooted at $\omega_0$)} if
 $\omega_0 = \omega_n$.  
  Note that we write $Q$ both for the edge
 weights (matrix entries) and for the induced measure
 on paths.

 We let $\laplace   =
I - Q$ denote the Laplacian.  
We write $G(x,y) = G^Q(x,y)$ for the Green's function that
can be defined either as
\[   G  = \Delta^{-1} =  \sum_{j=0}^\infty Q^j\]
or by
\[   G(x,y) =  Q[\paths_{x,y}(A)] =
\sum_{\omega \in \paths_{x,y}(A)} Q(\omega). \]
Provided   $Q$ is acceptable,
 these sums
converge absolutely.
We write 
\[    G(x,y) = G_R(x,y) + i \, G_I(x,y),\]
where $G_R,G_I$ are real matrices.
%The total variation of the measure $Q$ is given by
%\[ \|Q\|_{TV} = \|P \|_{TV} = \sum_{x,y \in A}
%   G^P(x,y) < \infty . \]
   
   Let
\begin{equation}  \label{sept4.1}
  f_x = \sum Q(\omega) 
  \end{equation}
where the sum is over all paths $\omega$ from $x$ to $x$ of
length at least one that have no other visits to $x$.  
A standard renewal argument shows 
that 
\begin{equation}  \label{sept4.2}
G(x,x) = \sum_{k=0}^\infty f_{x}^k,
\end{equation}
and since the sum is convergent,\[
|f_x|  < 1.\]  
If $V \subset A$, we will write 
\[ G_V(x,y) =
G^{Q_V}(x,y)  = \sum_{\omega \in \paths_{x,y}(V)}
    Q(\omega) , \]
for the corresponding Green's function associated to
paths in $V$.  The next proposition is
a well known relation between the determinant of the
Laplacian and the Green's function.

\begin{proposition}  \label{prop1}
If $A = \{x_1,\ldots,x_n\}$ and $A_j = A\setminus \{x_1,\ldots,x_{j-1}\}$,
\[   \frac{1}{\det \laplace} =  \prod_{j=1}^n
    G_{A_j}(x_j,x_j) . \]
 \end{proposition}
 
\begin{proof}  By induction on $n$.  If $n=1$ and $q =
Q(x_1,x_1)$,  there is
exactly one path of length $k$ in $A_{1}$ and it has
measure $q^k$.  Therefore
\[  G_{A_1}(x_1,x_1) = \sum_{k=0}^\infty  q^k = 
\frac 1{1-q}. \]
Assume the result is true for each $A_j\subsetneq A$, and note that if
$g(x) = G_{A_j}(x,x_j)$, then
\[    [I - Q_{A_j}] \, g = \delta_{x_j}
\]
 Using Cramer's rule to solve this linear system.
 we see that
 \[   G_{A_j}(x_j,x_j) = \frac{\det[I - Q_{A_{j+1}}]}
     { \det [ I - Q_{A_{j}}]}. \]

\end{proof}

\begin{proposition}  \label{prop2} If $Q$ is a Hermitian acceptable matrix, then for
each $x$, $G(x,x) >0$.  In particular, $\laplace$ and $G = \laplace^{-1}$
are   positive definite Hermitian
matrices.
\end{proposition}

\begin{proof}
It is immediate that $\Delta$ and $G$
are Hermitian.
 If $\omega$
is  a path in \eqref{sept4.1}, then so is $\omega^R$.  
Since $Q(\omega^R) = \overline{Q(\omega)}$, we can see
that  $\Im[f_x] =0 $, and hence  $-1 < f_x < 1$.
As in \eqref{sept4.2}, we can write
\[   G(x,x)   =
\sum_{k=0}^\infty f_x^k = \frac{1}{1 - f_x} >0.\]
Combining this with proposition \ref{prop1}, we see
that  each principal minor of $\Delta$
 is positive and hence $\Delta$ is positive definite. 
\end{proof}

 \section{Loop measures}

 \subsection{Definition}  \label{defsec}

 Let $\loopset = \loopset(A)$ denote the set
 of rooted loops of strictly positive length.
 If $Q$ is an acceptable weight, then 
the {\em (rooted) loop measure (associated to $Q$)} is the 
 complex measure
$m = m^Q$
on  $\loopset$, given by
\[          m(\omega) = \frac{Q(\omega)}{|\omega|}. \]
Note that the loop measure is {\em not}
the same thing as the path measure restricted to loops.  
An {\em unrooted loop} is an equivalence class of 
rooted loops  in $\loopset$ under
the equivalence relation generated by
\[  [\omega_0,\ldots,\omega_n] \sim [\omega_1,\ldots,\omega_n,\omega_1].\]
In other words, an unrooted loop is a loop for which
one forgets the ``starting point''.
We will write $\uloop$ for unrooted loops
and we let $\tilde \loopset$ denote the set
of unrooted loops.  We write $\omega \sim
\uloop$ if $\omega$ is in the equivalence class  $\uloop$.  The measure
$m$  induces  a measure that we call $\umeas$ by
\[   \umeas(\uloop) = \sum_{\omega \sim \uloop} m(\omega) . \]
We make several remarks.
\begin{itemize}
\item  Unrooted loops have forgotten their roots but have not 
lost their orientation.  In particular, $\uloop$ and $\uloop^R$
may be different unrooted loops.
\item  Since  $Q(\omega)$ and $ |\omega|$ are functions of the
unrooted loop, 
  we can write $Q(\uloop), |\uloop|$.  If $Q$ is Hermitian, then
$Q(\uloop^R) = \overline{Q(\uloop)}.$
\item Let $d(\uloop)$ denote the number of rooted loops
$\omega$ with $\omega \sim \uloop$. 
Note that  $d(\uloop)$ 
is an integer that divides $|\tilde \omega|$, but it is 
possible that $d(\uloop) < |\uloop|$.
 For example,
if $a,b,c$ are distinct elements and  $\uloop$ is the
unrooted loop with representative 
\[    \omega = [a,b,c,a,b,a,b,c,a,b,a] , \]
then $|\uloop| = 10$ and $d(\uloop)= 5$. Note that
\[   \umeas (\uloop) = \frac{d(\uloop)}{|\uloop|}
 \, Q(\uloop). \]
\item  Suppose that an unrooted loop $\uloop$ 
with $|\uloop| = n$ 
has $d = d(\uloop) $ rooted representative.  In other
words, the loop ``repeats'' itself after $d$ steps and
does $n/d$ such repetitions.
Suppose $k >0$ of these rooted representatives are
rooted at $x$. In the example above, $k=2$ for $x=a$
and $x=b$ and $k=1$ for $x=c$. 
  Then the total number of times that
the loop visits $x$ is $k(n/d)$.  Suppose that we give
each of the $k$ loops that are rooted at $x$ measure
$Q(\uloop)/[kn/d]$ and give all the other rooted representatives
of $\uloop$
measure zero.  Then the induced measure on unrooted loops
is the same as the usual unrooted loop measure, giving
measure $(d/n) \, Q(\uloop)$ to $\uloop$.

\item In other words,  
if we give each {\em  rooted} loop rooted at $x$ measure $Q(\omega)/k$
where $k$ is the number of visits to $x$, then the induced
measure on unrooted loops restricted to loops that intersect
$x$ is the same as $\umeas$. 

\item  One reason that the unrooted loop measure is
useful is that one can move the root around to do calculations.
The next lemma is an example of this.

\end{itemize}

Let 
\[  F(A) = F^Q(A) = \exp \left( 
 \sum _{\uloop \in \tilde \loopset} \umeas(\uloop)\right)
 = \exp \left( 
 \sum _{\omega \in  \loopset} m(\omega)\right). \]
If $V \subset A$, we let
\[   F_V(A) = \exp \left( 
 \sum_{\uloop \in \tilde \loopset, \uloop
\cap V \neq \eset }  \umeas(\uloop) \right).\]
  Note that $F_A(A) = F(A)$.
If $V = \{x\}$, we write just $F_x(A)$.  The next
lemma relates the Green's function to the exponential
of the loop measure; considering the case where $Q$ is positive shows that the sum converges absolutely.  As a corollary, we will have
a relationship between the determinant of the Laplacian
and the loop measure.

\begin{lemma}  \label{lemma2.1}
\[      F_x(A) =  G(x,x)  . \]
   More generally,
if $V = \{x_1,\ldots,x_l\}   \subset A$ and $A_j =A
\setminus \{x_1,\ldots,x_{j-1}\}$, then
\[    F_V(A)= \prod_{j=1}^l  G_{A_j}(x_j,x_j) 
    .\]

\end{lemma}

\begin{proof}  Let ${\mathcal A}_k$ denote the set of
$\uloop \in \tilde \loopset$   
that have $k$ different representatives that are rooted at $x$.
   By spreading the mass evenly
 over these $k$ representatives,  as described in the second and
 third to last bullets
 above, we can see that
 \[   \umeas\left[{\mathcal A}_k\right] = \frac{1}{k} \, f_x^k.\]
 Hence,
 \[ \umeas\left[ \bigcup_{k=1}^\infty 
 {\mathcal A}_k\right] = \sum_{k=1}^\infty \frac{1}{k} \, f_x^k
   = - \log [ 1 - f_x] = \log G(x,x) .\]
This gives the first equality 
and by iterating this fact, we get the second equality.

\end{proof}

\begin{corollary}  \label{2,2}
  \[   F(A) = \frac{1}{\det \laplace}. \]
      \end{corollary}

 \begin{proof}  Let $A  =  \{x_1,\ldots,x_n\}, A_j
  = \{x_j,\ldots,x_n\}$. 
   By proposition \ref{prop1} and lemma \ref{lemma2.1},
  \[   \frac{1}{\det \laplace}
   = \prod_{j=1}^n  G_{A_j}(x_j,x_j) = F(A). \]

  \end{proof}

 Suppose $f$ is a complex valued function defined on $A$
 to which we associate the diagonal matrix 
 \[        D_{f}(x,y) = \delta_{x,y}  \, f(x). \]
 Let $Q_f = D_{1/(1+f)}\,  Q$, that is, 
 \[   Q_f(x,y) =     \frac{Q(x,y)} {1 + f(x)}. \]
 If $Q$ is acceptable, then for $f$ sufficiently
 small,   $Q_{f}$ will be an acceptable matrix for which
 we can define the loop measure $m_f$.
   More specifically, if $\omega  =[\omega_0,\ldots,\omega_n] \in \loopset$,
  then
\[  m_f(\omega) = \frac{Q_f(\omega)}{|\omega|}
 = m(\omega) \, \prod_{j=1}^n \frac{1}{1 + f(\omega_j)} , \]
\[   \umeas_f(\uloop) = 
   \umeas(\uloop) \, \prod_{j=1}^n \frac{1}{1 + f(\omega_j)} .\]
 Hence, if $G_f = G^{Q_f}$,
$$\det G_f = \exp\left(\sum_{\omega \in \loopset} m_f(\omega)\right)
= \exp\left(\sum_{\uloop \in \tilde \loopset} \tilde
m_f(\uloop)\right).
$$

 \textbf{Example.}
Consider a
 one-point space $A = \{x\}$ with $Q(x,x) =q
  \in (0,1)$.  For each $n > 0$, there is exactly one
  loop $\omega^n$  of length $n$ with 
  $Q(\omega^n) = q^n, m(\omega^n) = q^n/n$. 
  Then,  $\Delta$ is the $1 \times 1$ matrix
  with entry $1-q$, 
  \[  G_A(x,x) = \sum_{n=0}^\infty q^n = \frac1{1-q} , \]
  and
  \[   \sum_{\omega \in \loopset} m(\omega)
      = \sum_{n=1}^\infty \frac{q^n}{n} = -\log[1-q]. \]

\subsection{Relation to loop-erased walk}

Suppose  $\overline A$ is a finite set, $A \subsetneq \overline A,
\p A = \overline A \setminus A$, and  
 $Q$ is a an acceptable matrix on $\overline A$.
 Let  $\paths(A)$
denote the set of paths $ \omega =[\omega_0,\ldots,\omega_n]$
  with $\omega_n  \in \p A $ and
$\{\omega_0,\ldots,\omega_{n-1} \} \subset A.$ 
For each path $\omega$, there exists a unique loop-erased
path $LE(\omega)$ obtained from $\omega$
by chronological loop-erasure as follows.
\begin{itemize}
\item  Let $j_0 = \max\{j: \omega_j = \omega_0 \}$.
\item  Recursively, if $j_k < n$, then $j_{k+1} =
 \max\{j: \omega_j = \omega_{j_k + 1} \}$.
 \item   If $j_k = n$, then 
$  LE(\omega) = [\omega_{j_0},\ldots,\omega_{j_k}].$
 \end{itemize}
 If $\eta = [\eta_0,\ldots,\eta_k]$ is a self-avoiding
 path in  $\paths(A)$, we define its loop-erased measure
 by
 \[     \hat Q(\eta;A) = \sum_{\omega \in \paths(A), \,
  LE(\omega) = \eta} Q(\omega). \]
  The loop measure gives a convenient way to describe
  $\hat Q(\eta;A)$.
  
  \begin{proposition}
  \[    \hat Q(\eta;A) = Q(\eta) \,  F_\eta(A).\]
\end{proposition}

\begin{proof}
We can decompose any path $\omega$ with $LE(\omega)
 = \eta$ uniquely as
 \[  l^0 \oplus[\eta_0,\eta_1] \oplus l^1 \oplus [\eta_1,\eta_2]
  \oplus \cdots  \oplus  l^{k-1} \oplus [\eta_{k-1},\eta_k] \]
 where $l^j$ is a rooted loop rooted at $\eta_j$ that is contained
 in $A_j := A \setminus \{\eta_0,\eta_1,\ldots,\eta_{j-1}\}$.
By considering all the possibilities, we see that the measure of
all walks with  $LE(\omega)
 = \eta$ is
\[   G_A(\eta_0,\eta_0) \, Q(\eta_0,\eta_1) \, G_{A_1}(
  \eta_1,\eta_1) \, \cdots\,  G_{A_{k-1}}(\eta_{k-1},\eta_{k-1})
   \, Q(\eta_{k-1},\eta_k), \]
   which can be written as
   \[ Q(\eta) \, \prod_{j=0}^{k-1}   G_{A_j}(\eta_j,\eta_j)
    = Q(\eta) \, F_\eta(A) .\]

 \end{proof}

There is a nice application of this to spanning trees.
 Suppose that $A
 = \{x_0,x_1,\ldots,x_n\}$ is the vertex set of a finite connected graph, 
 and let   $Q$  be the transition
probability for simple random walk on the graph, that is,
$Q(x,y) =1/d(x)$ if $x$ and $y$ are adjacent, where
$d(x)$ is the degree of $x$. 
Consider the following algorithm due to David
Wilson \cite{Wilson} to choose a
spanning tree from $A$:   
\begin{itemize}
\item Let  $ A' = \{x_1,\ldots,x_n\}$.
\item  Start a random walk at $x_1$ and run it until it
reaches $x_0$.  Erase the loops (chronologically)
 and add the edges
of the loop-erased walk to the tree.
\item  Let $x_j$ be the vertex of smallest index
 that has not been added
to the tree yet.  Start a random walk at $x_j$, let
it run until it hits a vertex that has been added to
the tree.  Erase loops and add the remaining edges to the tree.
\item Continue until we have a spanning tree.
\end{itemize}
It is a straightforward exercise using the
last proposition to see that for
any tree, the probability that it is chosen is exactly
\[     \left[\prod_{j=1}^n d(x_j)\right]^{-1} \,F(A')\]
 which by corollary \ref{2,2} can be written as
 \[     \left[ \det [I - Q_{A'}] \,  \prod_{j=1}^n d(x_j)\right]^{-1}
 =    \frac{1}{\det [D -  K]}.\]
Here $D(x,y) = \delta_{x,y} \, d(x)$ is the diagonal matrix of degrees and $K$
is the adjacency matrix, both restricted to $A'$.  (The
matrix $D - K$ is what graph theorists call the Laplacian.)
We can therefore conclude the following.
The second assertion  is a classical result due  to
Kirchhoff called the {\em matrix-tree theorem}.

\begin{theorem}  Every spanning is tree is equally likely
to be chosen in Wilson's algorithm.  Moreover, the
total number of spanning trees is 
$\det[D-K].$  In particular, $\det[D-K]$ does not depend on the
 ordering $\{x_0,\ldots,x_n\}$ of the vertices of $A$.
\end{theorem}

\section{Loop soup and Gaussian free field}

\subsection{Soups}

If $\lambda >0$, then the Poisson distribution on
$\N$  is given by
\[       q^\lambda(k) = e^{-\lambda} \, \frac{\lambda^k}{k!}. \]
We can use this formula to define the Poisson ``distribution''
 for  $\lambda \in \C$.  In this case $q^\lambda$ is a complex measure 
 supported on $\N$  with variation measure  $|q^\lambda|$
 given by 
\[  |q^\lambda|(k) =  |e^{-\lambda} |\,
 \frac{|\lambda|^k}{k!} = e^{-\Re(\lambda)}\,
 \frac{|\lambda|^k}{k!}  , \] and   
 total variation
 \[     \|q^\lambda \| =
   \sum_{k=0}^\infty |q^\lambda|(k)  
    = \exp\{|\lambda| - \Re(\lambda) \} \leq  e^{2|\lambda|}. \]
Note that 
 \[    \sum_{k=1}^\infty |q^\lambda|(k) 
  =   \exp\{|\lambda| - \Re(\lambda) \}  \, [1 - e^{-|\lambda|}]
   \leq  |\lambda| \, e^{2|\lambda|}. \]
  The usual 
  convolution formula $q^{\lambda_1} * q^{\lambda_2} =
  q^{\lambda_1 + \lambda_2}$   holds, and if
  \[         \sum_{j=1}^\infty |\lambda_j| < \infty , \]\
  we can define the infinite convolution
  \[              \prod_{j}^*  q^{\lambda_j} =
    \lim_{n \rightarrow \infty} (q^{\lambda_{1}} * \cdots *
     q^{\lambda_{n}}) =q^{\sum \lambda_j}. \]

If $\lambda >0$ and $M_t$ is a Poisson process
with parameter $\lambda$, then the distribution
of $M_t$ is
\begin{equation} \label{100.1}
  q_{t}(\{k\}) = q^{t\lambda}(k) =  e^{-t\lambda} \, \frac{(t\lambda)^k}
    {k!}, \;\;\;\;k=0,1,2,\ldots
    \end{equation}
The family of measures $\{q_t\}$
 satisfy the semigroup law
$    q_{s+t} = q_s * q_t .$
If we are only interested in the measure $q_t$, then
we may choose $\lambda$ in \eqref{100.1} to be complex.
In this case the measures $\{q_t\}$ are not probability
measures but they still satisfy the semigroup law.
We call this the Poisson semigroup of measures with
parameter $\lambda$ and note that the Laplace
transform is given by   
$$
\sum_{k=0}^\infty e^{k\alpha } q_t(\{k\}) = \exp(t\lambda(e^\alpha -1)).
$$

Suppose $m$ is a complex measure on a countable set
$X$, that is, a complex function with
\[   \sum_{x \in X} |m(x)| < \infty.\]
Then we say that the soup generated by $m$  is the semigroup of
measures $\{q_t: t \geq 0\}$ on $\N^X$ where
$q_t$  is the product
measure  of $\{q_t^x: x \in X\} $ where  $\{q_t^x:
t \geq 0\}$ is
a Poisson semigroup of measures with parameter $m(x)$. Pushing forward $q_t$ along the map $\vf\mapsto\sum_{x\in X}\vf(x)$ to a measure on $\N\cup\{\infty\}$, we se that it agrees with $\prod^* q_t^x$ on $\N$ and thus $q_t$ is supported on the pre-image of $\N$, the set of $\vf \in 
\N^X$ with finite support which we will call $\N^X_{\rm fin}$.
The complex measure $q_t$ satisfies
\[  \|q_t\| \leq \prod_{x \in X} \|q_t^x\| \leq 
 \exp \left\{2t\sum_{x \in X} |m(x)| \right\}. \]
 
 Soups were originally
  defined when $m$ is a positive
 measure on $X$, in which case it is defined as an
 independent collection of Poisson processes $\{M_t^x:
  x \in X\}$ where $M_t^x$ has rate $m(x)$.  A
  realization $\ls_t$ of the soup at time $t$ is
  a multiset of $X$ in which the element $x$ appears
  $M_t^x$ times. In this case $q_t$ gives the
  distribution of the vector $(M_t^x: x \in X)$.

\subsection{Loop soup}

Suppose $Q$ is an acceptable weight with associated
loop measure $m$.  Let $0 < \epsilon < 1$  be such that
the matrix  with entries $ P_\epsilon(x,y) :=
e^\epsilon\,  |Q(x,y)|$ is
still acceptable.  Let $m$ be the rooted loop measure
associated to $Q$ and note that
\begin{equation}  \label{onion}
  \sum_{\omega \in \loopset} |\omega| \, e^{\epsilon |\omega|}
\, |m(\omega)| = \sum_{\omega \in \loopset} P_\epsilon(\omega) < \infty.
\end{equation}

The {\em(rooted) loop soup} is a ``Poissonian
 realization'' of the measure $m$.
To be more precise, recall that $\loopset$ is  the set of
rooted loops in $A$ with positive length.   A 
  multiset $\ls$ of loops is a  generalized
  subset of $\loopset$
   in which
  loops can appear more than once.  
  In other words 
  it is  an element $\{  \ls(\omega) :
  \omega \in \loopset\}$
of $\N^\loopset$  
where $\ls(\omega)$ denotes the number of times that 
$\omega$ appears in $\ls$. 
Then the  rooted loop soup
is the semigroup of measures $\lmeas_t =
\lmeas_{t,m}$ on $\N^\loopset$ given by the product
measure of the Poisson semigroups $\{\lmeas^\omega_t
: \omega \in \loopset\}$ where
$\lmeas^\omega_t$ has parameter $m(\omega)$. 
The measure $\lmeas^\omega_t$ is Poisson with
parameter $tm(\omega)$ and hence $\|\lmeas^\omega_t\|
\leq \exp\{2t|m(\omega)| \} $ and
\begin{equation}  \label{radish}
\sum_{k=1}^\infty| \lmeas^\omega_t(k) |
 \leq t |m(\omega)| \, e^{2t|m(\omega)|}.
 \end{equation}
%We claim that  $|\lmeas_t|,$ and hence $\lmeas_t$, is supported on
%$\N^\loopset_{\rm fin}$  Indeed, using
%\eqref{radish}, we see that if $O \subset \loopset$,
%\[ 
 % |\lmeas_t|\{\ls; \ls(\omega) > 0 \mbox{ for some }
%\omega \in O\}  \leq  t \, \left[\sum_{\omega \in O}
 % |m(\omega)|\right] \,  \exp \left \{2t\sum_{\omega \in \loopset}
 %| m(\omega)|\right\} . 
 %\]

For any  $x \in A$ and rooted loop
 $\omega $, we define the {\em (discrete) local time} $N^\omega(x)$ to be the number of visits of $\omega$ to $x:$
$$
N^\omega(x) = \sum_{j=0}^{|\omega|-1}1\{\omega_j=x\}
 = \sum_{j=1}^{|\omega|} 1\{\omega_j=x\}.
$$
Note that this is a function of an unrooted loop, so we can
also write $N^\uloop(x)$.   Also $N^{\omega^R}(x) =N^\omega(x)$.
We define the additive function  
  $L:\N^\loopset_{\rm fin} \rightarrow \N^A$ by  
\[   L_\ls(x) =  \sum_{\omega \in \loopset} 
\ls(\omega)  \, N^\omega(x) 
  . \] 
By pushing forward by $L$, the loop soup  $\lmeas_t$ induces
a measure on $\N^A$ which we denote by
$\mu_t = \mu_{t,m}$ and  refer to as the {\em discrete
occupation field.}   Indeed, since $\lmeas_t$ is a product
measure,  we can write $\mu_t$ as
\[  \mu_t = \prod^*_{\omega \in \loopset}  \mu_t^\omega\]
where the notation $\prod^*$  means convolution
and $\mu_t^\omega$ denotes the measure supported on
$\{kN^\omega: k=0,1,2,\ldots\}$ with 
\[   \mu^\omega_t (k N^\omega) =
       e^{-t m(\omega) } \, \frac{[tm(\omega)]^k}{k!}.\]
For future reference we note that since $N^\omega =
N^{\omega^R}$,
\[
   [\mu^\omega_t * \mu^{\omega^R}_t]
    (k N^\omega) =
       e^{-t [m(\omega) + m(\omega^R)] } \, \frac{t^k
        \, [m(\omega) + m(\omega^R)]^k}{k!},
      \]
and hence, 
\begin{equation}  \label{feb28.8}
 \mu_{2t} =  \prod^*_{\omega \in \loopset}  \mu_{2t}^\omega
  = \prod^*_{\omega \in \loopset}  \mu_{t}^\omega * \mu_t
 ^\omega = \prod^*_{\omega \in \loopset}  \mu_{t}^\omega * \mu_t
 ^{\omega^R}  = \mu_{t,m^R}, 
 \end{equation}
 where 
 \[       m^R(\omega) = m(\omega) + m(\omega^R) . \]

\subsection{A continuous occupation field}

In order to get a representation of the Gaussian
free field, we need to  change the discrete occupation field to a continuous
time occupation field.  We will do so in a simple way by replacing $N^\omega(x)$ with a sum of $N^\omega(x)$ independent rate one exponential random
variables.
  This is similar to the method of constructing
continuous time Markov chains from discrete time chains by adding exponential 
waiting times.

We say that a process $Y(t)$ is a {\em gamma process} if it has independent increments, $Y(0)=0$, and for any $t,s\geq 0,$ $Y(t+s)-Y(t)$ has a Gamma$(s,1)$ distribution. In particular, $Y(n)$ is distributed as the sum of $n$ independent rate one exponential random variables.
Let
$\{Y^x: x \in A\}$ be a collection of independent gamma processes.
If $\bar s =\{s_x: x \in A\} \in [0,\infty)^A$, we write
$Y(\bar s)$ for the random vector
$(Y^x(s_x))$.  The Laplace transform
  is well known,
\[   \E\left[\exp\{- Y(\bar s)\cdot f\}
 \right] = \prod_{x \in A} \frac{1}{[1+f(x)]^{s_x}} , \]
provided that $\|f\|_\infty < 1$.  In particular,
if $\ls \in \N^\loopset_{\rm fin}$, then
\begin{eqnarray}
  \E\left[\exp\{- Y( L_\ls)\cdot f \}
 \right] 
 &  =& \prod_{x \in A} \frac{1}{[1+f(x)]^{L_\ls(x)}}\nonumber\\
  &  = & \prod_{\omega \in \loopset}
     \prod_{x \in A} \frac{1}{[1 + f(x)]^{\ls(\omega)
      \, N^\omega(x) }}\nonumber  \\
     & = & \prod_{\omega
\in \loopset}\exp\left[-
\ls(\omega)(\ln(1+f)
\cdot N^\omega)
\right].  \label{feb28.1}
\end{eqnarray}
  
%We will now suppose that the gamma process $Y$
%and the loop soup $\ls_t$ are chosen independently to give
%a continuous occupation field.  

For positive $Q$, we could then define a continuous occupation field in terms of random
variables, and we let $\loops_t = Y(L_{\ls_t})$ by taking $\ls_t$ as an independent loop soup corresponding to $|Q|$. In order  to handle the
  general case,  we  define
the ``distribution'' of the continuous occupation
field at time $t$ to be
 the complex measure $\nu_t = \nu_{t,m}$  on $[0,\infty)^A$ 
given by 
\begin{equation}  \label{feb28.6}
    \nu_t(V)  =  \sum_{\ls \in \N^\loopset_{\rm fin} } \lmeas_{t}(\ls)
  \, \Prob\{Y(L_\ls) \in V \} 
  = \sum_{\bar k \in \N^{A}}
   \mu_t(\bar k) \,  \Prob\{Y(\bar k) \in V \},
    \end{equation}
   where  $
    V \subset [0,\infty)^A.$
We will write \[ \nu_t[h(\loops)]   = \int_{[0,\infty)^A}h(\loops)d\nu_t(\loops)\]
 provided
that 
$\int_{[0,\infty)^A}|h(\loops)| \, d|\nu_t|(\loops)
< \infty.$
\begin{lemma}
If $\E[|h(\loops_t)|]<\infty$, then $|\nu_t|[|h(\loops)|]<\infty$.
\end{lemma}
\begin{proof}
First, note that
$$
|\lmeas_t(\ls)| = \left |\prod_{\w\in\loopset}\lmeas_t^\w(\ls(\w))\right| = \left|e^{-t\sum_{\w\in\loopset} m(\w)}\right|\prod_{\w\in\loopset}\frac{(t|m(\w)|)^{\ls(\w)}}{\ls(\w)!}= \alpha\Prob\{\ls_t=\ls\}
$$
with $\alpha=e^{t\sum_{\w\in\loopset}|m(\w)|-\Re(m(\w))}$. Thus, taking $\sup$ over all finite partitions $\{V_i\}_{i=1}^n$ of $V$ into measurable sets,
\begin{eqnarray*}
|\nu_t|(V) &=& \sup \sum_{i=1}^n |\nu_t(V_i)| = \sup \sum_{i=1}^n \left|\sum_{\ls \in \N^\loopset_{\rm fin} } \lmeas_{t}(\ls)
  \, \Prob\{Y(L_\ls) \in V_i \} \right|\\
  &\leq&   \sup \sum_{i=1}^n  \sum_{\ls \in \N^\loopset_{\rm fin} }| \lmeas_{t}(\ls)|
  \, \Prob\{Y(L_\ls) \in V_i \}\\
  &=&  \sum_{\ls \in \N^\loopset_{\rm fin} }| \lmeas_{t}(\ls)|
  \, \Prob\{Y(L_\ls) \in V\} = \alpha \Prob\{\loops_t \in V\}. 
\end{eqnarray*}
\end{proof}

We will compute the Laplace transform of the measure
$\nu_t$, but first we need use the following lemma.
%\begin{comment}
\begin{lemma}   \label{pepperlemma}
Suppose  $S$ is a countable set and 
 $F:S\times \N\to\C$  is a function with   $F(s,0)=1$ for all $s\in S$,
 \[   \sum_{s \in S}\left| \sum_{n=1}^\infty F(s,n)\right | < \infty  \]
 and
 $$
 \sum_{\psi\in\N_{\rm fin}^S} \left| \prod_{s\in S} F(s,\psi(s)) \right|<\infty.
 $$
 Then,
$$
\prod_{s\in S}\sum_{n=0}^\infty F(s,n) = \sum_{\psi\in \N_{\rm fin}^S}\prod_{s\in S} F(s,\psi(s)).
$$ 
\end{lemma}
\begin{proof}
Since $\sum_{s \in S} |\sum_{n=1}^\infty F(s,n)| < \infty$, the product on the left-hand side does not 
depend on the order.  For this reason we may assume that
$S$ is the positive integers and  write
\begin{eqnarray*}
  \prod_{s=1}^\infty \sum_{n=0}^\infty F(s,n)
 & =&  \lim_{J \rightarrow \infty}  \prod_{s=1}^J \sum_{n=0}^
 \infty  F(s,n)\\
 & = & \lim_{J\to\infty} \sum_{\psi\in \N^{J}}\prod_{s=1}^JF(s,\psi(s)) 
  =  \sum_{\psi\in \N_{\rm fin}^\infty}\prod_{s=1}^\infty F(s,\psi(s)).\\
 \end{eqnarray*}
 The last equality uses the absolute convergence of the final sum.
 \end{proof}

%Enumerate $S$ as $S=\{s_j\}_{j=0}^\infty$ and let $S_J= \{s_J\}_{j=0}^J$. Expand the left hand side:
%\begin{eqnarray*}
%\prod_{j=0}^\infty \sum_{n=0}^\infty F(s_j,n) &= &
%\lim_{J\to\infty} \prod_{j=0}^J \sum_{n\in\N} F(s_j,n)
%\\
%&= & \lim_{J\to\infty} \sum_{\psi\in \N^{S_J}}\prod_{j=0}^JF(s_j,\psi(s_j))
%\\
%&= &\lim_{J\to\infty} \sum_{\psi\in \N^{S_J}}\prod_{s\in {\rm supp}\psi}F(s,\psi(s))\\
%&= &\sum_{\psi\in\N^S_{\rm fin}}\prod_{a\in {\rm supp}\psi} F(s,\psi(s)) =\sum_{\psi\in\N^S_{\rm fin}}\prod_{s\in S} F(s,\psi(s)).
%\end{eqnarray*}
%\end{proof}
%\end{comment}

\begin{proposition}\label{mgfloops} For $f$ sufficiently 
small,
\begin{equation}  \label{feb26.3}
\nu_t[\exp(-\loops  \cdot f)] = \left(\frac{\det G_f}{\det G}\right)^{t}.
\end{equation}
\end{proposition}

%Note that  if $\|f\|_\infty < \delta$, 
%\[ \E\left[| \exp\{-Y(L_C) \cdot f \}   | \right]\leq 
%        \prod_{\omega \in \loopset} 
%          e^{ \delta \, |\omega|\,C(\omega)}, \]
%           and hence
%     \[   |\nu_t| \left(|\E[ \exp\{-Y(L_C) \cdot f \}   | \right)
%           \leq \prod_{\omega \in \loopset} e^{2 |m(\omega)|}\,
%                      \E  [e^{C_\omega \,|\omega| \, \delta}] , \]
%              where $C_\omega $ denotes 
%  a Poissson random variable with parameter $|m(\lambda)|$.
%  The right-hand side equals
%  \[ \exp \left\{2 \sum_{\omega \in \loopset} |m(\omega)|
%   \right\} \, \exp \left\{t\sum_{\omega \in \loopset}
%   m(\omega) \,(e^{\delta |\omega|} - 1) \right\}, \]
%   which is finite for $\delta$ sufficiently small by \eqref{onion}.
%   

\begin{proof}
We first claim that there exists $\delta > 0$ such that
if $\|f\|_\infty < \delta$, 
%\begin{equation}\label{pepper}
 $$  \E[ |\exp\{-\loops_t \cdot f \}   |  ]
     < \infty,
 $$%    \end{equation}
so that the left hand side of \eqref{feb26.3} is well defined. Indeed, if $\|f\|_\infty < \delta$,
 and $  (1-\delta)  =e^{-\epsilon}$,  then for any $\ls\in\N^\w_{\rm fin}$\[ \E\left[| \exp\{-Y(L_{\ls}) \cdot f \}   | \right]\leq  \prod_{\omega \in \loopset}
     |1 - \delta|^{-\ls(\omega) \, |\omega|} 
= 
        \prod_{\omega \in \loopset} 
          e^{ \epsilon \, |\omega|\,\ls(\omega)}, \]
           and hence
\begin{eqnarray*}
 \E[ |\exp\{-Y(L_{\ls_t}) \cdot f \}   |  ] &=& \E\left[\E[|\exp\{-Y(L_{\ls_t}) \cdot f \}   | \big | \ls_t]\right]\\
 &\leq& \E\left[\prod_{\w\in \loopset} e^{\epsilon |\w|\ls_t(\w)}\right] \\
 &=& \prod_{\w\in\loopset} \exp\left(t|m(\w)|(e^{\epsilon|\w|}-1)\right)
\end{eqnarray*}           
%     \[   |\nu_t| \left(|\E[ \exp\{-Y(L_C) \cdot f \}   | \right)
%           \leq \prod_{\omega \in \loopset} e^{2 |m(\omega)|}\,
 %                     \E  [e^{C_\omega \,|\omega| \,\epsilon}] , \]
 %             where $C_\omega $ denotes 
 % a Poisson random variable with parameter $t|m(\lambda)|$.
%  The right-hand side equals
 % \[ \exp \left\{2 \sum_{\omega \in \loopset} |m(\omega)|
 %  \right\} \, \exp \left\{t\sum_{\omega \in \loopset}
  % |m(\omega) | \,(e^{\epsilon |\omega|} - 1) \right\}, \]
   which is finite for $\epsilon$ sufficiently small by \eqref{onion}.

We assume that $\|f\|_\infty < \delta $.
Using \eqref{feb28.1} we get 
\begin{eqnarray*}
\nu_t[\exp(-\loops  \cdot f)] &
= & \sum_{\ls\in \N^\loopset_{\rm fin}} \lmeas_t(\ls)
 \, \E\left[\exp(-Y(L_\ls)  \cdot f) \right]\\
& = &  \sum_{\ls\in \N^\loopset_{\rm fin}} \prod_{\omega
\in \loopset} \lmeas^\omega_t(\ls(\omega))\exp\left[
\ls(\omega)(-\ln(1+f)
\cdot N^\omega)
\right] \\
& = & \prod_{\omega\in \loopset}\sum_{n=0}^\infty \lmeas^\omega_t(n)\exp\left[
n(-\ln(1+f)
\cdot N^\omega)
\right]\\
%& = &\prod_{\omega
%\in \loopset} \lmeas_t^\omega\left(\exp\left[
%\ls(\omega)(-\ln(1+f)
%\cdot N^\omega)
%\right]\right)\\
&= &\prod_{\omega
\in \loopset}\exp\left[
tm(\omega)(\exp(
-\ln(1+f)
\cdot N^\omega
)-1)
\right].\\
\end{eqnarray*}
The third equality uses lemma \ref{pepperlemma}, which is valid as
$$
\sum_{\w\in\loopset}\left|
\sum_{n=1}^\infty \frac{(tm(\w))^n}{n!}\exp\left[n(-\ln(1+f)\cdot N^\w)\right]
\right|= \sum_{\w\in\loopset} \left|
\exp \left[tm_f(\w)\right]
-1\right|<\infty.
$$
The equality used
$$\exp(-\ln(1+f)\cdot N^\omega) = \prod_{x\in A} \exp \ln \left[\left(\frac 1 {1+f(x)}\right)^{N^\omega(x)}\right] = \prod_{j=0}^{|\omega|-1} \frac 1{1+f(\omega_j),}$$
which also gives us
\begin{eqnarray*}
\nu_t[\exp(-\loops  \cdot f)] 
&=&
{\exp\left[t
\sum_{ \omega
\in \loopset}m(\omega)
\prod_{j=0}^{|\omega|-1}\frac 1{1+f(\omega_j)}
\right]}\,
{\exp	 \left[-t\sum_{\omega
\in \loopset}m(\omega)
\right]} \\
& = & \exp\left[t\sum_{\omega
\in \loopset}
 m_f(\omega) \right]
  \, {\exp	 \left[-t\sum_{\omega
\in \loopset}m(\omega)
\right]} \\
&=&
\left(\frac{\det G_f}{\det G}\right)^{t}.
\end{eqnarray*}
The last equality is by corollary \ref{2,2}.
\end{proof}

Consider the one-point example at the end of section
\ref{defsec}.   If we let $s$ denote the function $f$ taking
the value $s$, then
\[   m_s(\omega^n) =  \frac{1}{n} \, \left[\frac{q}{1 + s}\right]^n,\]
and
\[    \sum_{\omega \in \loopset}  
 m_s(\omega^n)  = \sum_{n=1}^{\infty}
 \frac{1}{n} \, \left[\frac{q}{1 + s}\right]^n
    = -\log \left[1 - \frac{q}{1+s}\right] . \]
  Therefore, if $\loops_t$ denotes the continuous time
  occupation field at time $t$,
\begin{equation}  \label{mar11.1}
\E\left[e^{-s \loops_t}\right] = \left(\frac{\det G_s}{\det G}\right)^{t}=
   \left[\frac{1 + s - q(1+s)}
   {1 + s - q } \right]^{t}.
   \end{equation}

We recall that we have defined the (discrete time)
loop measure $m(\omega) = Q(\omega)/|\omega|$ and then
we have added continuous holding times.   Another
approach, which is the original one taken by Le Jan \cite{LeJan}, is
to construct a loop measure on continuous time paths.
Here we start with $Q$, add the waiting times to give a measure
on continuous time loops, and then divide the measure by
the (continuous) length.  Considered as a measure on
unrooted continuous time loops, the two procedures are 
essentially equivalent (although using discrete time loops makes it
easier to have ``jumps'' from a site to itself).

\subsection{Trivial loops}

We will see soon that the loop soup and the square of the Gaussian free field are closely related, but because our construction of the loop soup used discrete loops and only added continuous time afterwards, we restricted our attention to loops of positive length.
 We will
need to add a correction factor to the occupation time to account for these  trivial loops  which are formed by viewing the continuous time process before its first jump.

Consider the one-point example at the end of section
\ref{defsec}.  The Gaussian free field with covariance
matrix $[I - Q]^{-1}$ is just a centered normal random variable
 $Z$  with variance $1/(1-q)$  which we
 can   write as $  N/\sqrt{1-q}$ where
  $N$ is a standard normal.   Since  $N^2$ has a $\chi^2$
  distribution with one degree of freedom, we see that 
  \[        \E\left[e^{-sZ^2/2} \right] = \sqrt{\frac{1-q}{1-q+s} } \]
 If we compare this to \eqref{mar11.1}, we can see that
\[   \E\left[e^{-sZ^2/2} \right]  =
  \E\left[e^{-s \loops_{\frac12}}\right]  \, 
  \left[1 + \frac{s}{1-q}\right]^{-1/2} 
  .\]
  The second term on the right-hand side is the
  moment generating function for a  Gamma$(\frac 12,1)$
  random variable.  Hence we can see that $Z^2/2$ has the
  same distribution at $\loops_{\frac 12} + Y$ where $Y$
  is an independent  Gamma$(\frac 12,1)$ random
  variable. 

  The trivial loops we will add 
are not treated in the same way as the other loops.
To be specific, we add another collection of independent gamma processes $\{Y_{\text{trivial}}^x\}_{x\in A}$ and define the {\em occupation field of the trivial loops} as
$$
\trivial_t(x)= Y^x_{\text{trivial}}(t).
$$
When viewed in terms of the discrete time loop measure,
this seems unmotivated.  It is useful to consider
the  
  continuous time loop measure in terms of continuous
  time Markov chains. 
  For any $t$ prior to the first jump of the Markov chain, the path will form a trivial loop of time duration $t$. As the Markov chain has exponential holding times, the path measure (analogue of $Q$) to assign to such a trivial loop is $e^{-t}\, dt$, and so the loop measure (analogue of $m$)
  should be $ t^{-1} \, {e^{-t}} \, dt$.  Hence in the
  continuous time measure, we give trivial loops of
  time duration $t$ weight  $   {e^{-t}} /t$.
Since  $  {e^{-t}}/t $ is the intensity measure for the jumps of a gamma process,  we see that the 
added occupation time at $x$  corresponds to $\trivial_t(x)$.

We write $\nu_t^\trivial$ for the probability distribution
of $\trivial_t$.  In other words, it is the distribution
of independent gamma processes $\{Y^x(t): x \in A\}$.
Note that if  $\loops \in [0,\infty)^A$, 
\begin{equation}  \label{feb28.7}
   \nu_t^\trivial[\exp\{-f \cdot \loops\}]
 =    \prod_{x \in A} \frac{1}{[1 + f(x)]^t}
  = [\det D_{1 + f}]^{-t}   . 
  \end{equation}
We will also write
\[           \rho_t =\nu_t * \nu_t^\trivial, \]
which using \eqref{feb28.6}
 can also be written as
\[  \rho_t(V)  =  \sum_{\ls \in \N^\loopset_{\rm fin} } \lmeas_{t}(\ls)
  \, \Prob\{Y(L_\ls + \bar t) \subset V \}, \]
  where $\bar t$ denotes the vector each of whose components
  equals $t$.

\subsection{Relation to the  real Gaussian free field}

If $A$ is a finite set with $|A| = n$,
and $G$ is a  symmetric,
 positive definite  real  matrix,
then the 
 {\em (centered, discrete) Gaussian free field} on $A$ 
 with covariance matrix $G$
  is the random function  $\vf:A\to\R$, defined by having density 
  \[    \frac{1}{(2 \pi)^{n/2} \, \sqrt{\det G}}\, 
             \exp\left(-\frac12\vf \cdot G^{-1} \vf\right)
              =   \frac{1}{(2 \pi)^{n/2} \, \sqrt{\det G}}\, 
             \exp\left(-\frac12|J \vf |^2  \right)\]
with respect to Lebesgue measure on $\R^n$.  Here $J$
is a positive definite, symmetric square root of $G^{-1}$. 
In other words, $\vf$ is a $|A|$-dimensional 
mean zero 
normal random variable with covariance matrix $\E[\vf(x)\vf(y)]=G(x,y)$.

\begin{lemma}  Suppose $G$ is a symmetric  
 positive
definite  matrix and let $\phi$ denote a Gaussian
free field with covariance matrix $ G$.   
Then for all $f $ sufficiently small,
 \begin{equation}  \label{feb24.2}
\E\left[\exp\left(-\frac 12 \vf^2\cdot f\right)\right]  = \frac{1}{\sqrt{\det{(\laplace+D_f)}}}\,
\frac1 {\sqrt{\det{G}}}.
\end{equation}
\end{lemma}

\begin{proof}  This is a standard calculation,
\begin{eqnarray*}
\E\left[\exp\left(-\frac 12 \vf^2\cdot f\right)\right] &=& \frac 1{(2 \pi)^{n/2} \, \sqrt{\det G}} \int_{\R^n} \exp\left(-\frac 12 \vf^2\cdot f\right)\exp\left(-\frac12\vf \cdot G^{-1}\vf\right) d\vf\\
&=& \frac 1{(2 \pi)^{n/2} \, \sqrt{\det G}}\int_{\R^n} \exp\left(-\frac 12 \vf \cdot (\laplace+D_f)\vf\right)d\vf.\\
\end{eqnarray*} 
If $f$ is  sufficiently small, then   $\laplace+D_f$  
  is a   positive definite symmetric
   matrix and so has a positive
   definite  square root, call it $R_f$. Then
\begin{eqnarray*}
\int_{\R^n} \exp\left(-\frac 12 \vf \cdot (\laplace+D_f)\vf\right)d\vf  &= &\int_{\R^n} \exp\left(-\frac 12 R_f\vf \cdot R_f\vf\right)d\vf
\\
&=  &\frac 1{\det R_f}\int_{\R^n} \exp\left({-\frac12\vf\cdot\vf}\right)d\vf\\
& = & \frac{(2 \pi)^{n/2}}{\sqrt{\det (\Delta + D_f)}} 
 .
\end{eqnarray*}

 \end{proof}

\begin{theorem}\label{gffLoop}
If $Q$ is a symmetric, acceptable
real matrix, and $\vf$ is the discrete Gaussian free field on $A$ 
with covariance matrix  $G $, then the distribution
of $\frac 12 \vf^2$ is $\rho_{\frac12} $.
 
\end{theorem}

\begin{proof}
It suffices to  show that the Laplace transforms for $\frac 12 \vf^2$ and $\rho_{\frac 12 }$ exist and
agree on a neighborhood of zero.   We
have calculated the   transforms   for  $\loops_{\frac12}$ and
$\vf^2$ in \eqref{feb26.3}  and \eqref{feb28.7}
giving

\begin{eqnarray*}
\rho_{\frac 12}
 \left[\exp\{\loops \cdot
f \} \right] & = & \nu_{\frac12} \left[\exp\{\loops \cdot
f \} \right] \, \nu_{\frac 12}^\trivial \left[\exp\{\loops \cdot
f \} \right] 
\\&=&\det(D^{-1}_{1+f})^{\frac 12}\left(\frac{\det G_f}{\det G}\right)^{\frac 12}\\
&= &\det(D^{-1}_{1+f})^{\frac 12}\left(\frac{\det(I-D_{1+f}^{-1}Q)^{-1})}{\det G}\right)^{\frac 12}\\
&=&\left(\frac{\det(D_{1+f}-Q)^{-1}}{\det G}\right)^{\frac 12}\\
&=& \frac{1}{\sqrt{\det{(\laplace+D_f)}}}\,
\frac1 {\sqrt{\det{G}}}.
\end{eqnarray*}
Comparing this to \eqref{feb24.2} completes the proof.
\end{proof}

Conversely, suppose that a symmetric,
positive definite
real matrix $G$ is given, indexed by the elements
of $A$ and let $\{\phi(x): x \in A\} $ denote
the Gaussian free field.
 If the matrix $Q := I - G^{-1}$ is positive
definite and acceptable, then we can use loops to give
a representation of $\{\phi(x)^2: x \in A\}$. 
If $G$ has negative entries then so must $Q$ (since the
Green's function for positive weights is always positive).
% Although
%it is not always the case that $Q$ as defined above
%is positive definite and acceptable, this will be true for
%\[      Q_\lambda := I - (\lambda G)^{-1} \]
%for $\lambda$ sufficiently large.  If $\phi_\lambda$
%denotes the associated field, then the distribution
%of $\lambda^{-1} \, \phi_\lambda^2$ is the same
%as that of $\phi^2$. 

\subsection{Complex weights}

There is also a relation between complex, Hermitian
weights and a complex Gaussian field. 
Let $A$ be a finite set with $n$ elements.
Suppose $G'$ is a positive definite Hermitian matrix and
let $K$ be a positive definite Hermitian square root of $(G')^{-1}$. 
The {\em (centered) complex Gaussian free field
on $A$  with covariance
matrix $ G'$}
 is defined to  
 be the measure on complex
functions 
$h: \R^A \rightarrow \C$ with
density
\[    \frac{1}{ \pi^{n } \,  {\det G'}}\, 
             \exp\left(-\overline 
             h \cdot  (G')^{-1} h\right)
              =   \frac{1}{ \pi^{n } \,  {\det G'}}\, 
             \exp\left(-|Kh|^2
          \right)\]
  with respect to Lebesgue measure on $\C^{n}$
  (or $\R^{2n}$).
 Equivalently, the function
 $\psi =\sqrt 2 \, h$ has density
\begin{equation}  \label{cdensity}
    \frac{1}{(2 \pi)^{n } \,  {\det G'}}\, 
             \exp\left(-\frac12\overline 
             \psi \cdot  (G')^{-1} \psi\right) = \frac{1}{(2 \pi)^{n } \,  {\det G'}}\, 
             \exp\left(-\frac12  \left|K\psi\right|^2
               \right)
              , 
              \end{equation}  
              It satisfies the
              covariance relations
\begin{equation}  \label{feb24.5}
\E \left[\overline h(x) \, h(y) \right]
 = \frac 12 \,  \E \left[\overline{\psi (x) } \, \psi(y) \right]
    = G'(x,y) ,
    \end{equation}
    \[  \E \left[ h(x) \, h(y) \right] =0.\]
The complex Gaussian free field on  a set of $n$
elements can be considered
as a  real field on $2n$ elements by viewing the
real and imaginary parts as separate components.
The next proposition makes this precise.
   Let $A^* = \{x^*: x \in A\}$ be another copy of $A$ and $\overline A
= A \cup A^*$. We can view $\overline A$ as a ``covering
space'' of $A$ and let $\Phi:\overline A \rightarrow A$
be the covering map, that is, $\Phi(x) = \Phi(x^*) = x$.
We call $A$ and $A^*$ the two ``sheets'' in $\overline A$.
Let $G'= G_R + i G_I$ and define $G$ on $\overline{A}$
by
\[    G(x,y) = G(x^*,y^*) = G_R(x,y) , \]
\[   G(x,y^*) =-G(x^*,y) =-G_I(x,y) . \]
Note that $G$ is a real, symmetric, positive
definite matrix.

\begin{proposition}  \label{water}
  Suppose $G' = G_R + i G_I$
is a positive definite Hermitian matrix indexed
by $A$ and  suppose  $G$  is the positive definite,
symmetric matrix indexed by $\overline A$, 
      \[ G= \bordermatrix {~ & A & A^* \cr
           A & G_R & - G_I \cr
           A^* & G_I & G_R\cr} . \]
Let $\{\phi_z : z \in \overline A\}$ be a centered Gaussian free field on $\overline A$
with covariance matrix $G$.  If
\begin{equation}  \label{cfield1}
\psi_x = \phi_x + i \, \phi_{x^*} , 
\end{equation}
then $\{\psi_x: x \in A\}$
is a complex centered Gaussian free field with covariance
matrix $2G'$.
\end{proposition}

\begin{proof}  Let $K = K_R + i K_I$ be
the Hermitian positive definite square root
of $(G')^{-1}$  and write $(G')^{-1} = \Delta_R +
 i \, \Delta_I$.  The relation $K^2 = (G')^{-1}$
 implies
\[   K_R^2 - K_I^2 = \Delta_R, \;\;\;\;
   K_R \, K_I + K_I \, K_R = \Delta_I. \]
and $G'\, (G')^{-1} = I$ implies
\[    G_R \, \Delta_R - G_I \, \Delta_I 
 = I, \;\;\;\;\;  G_R \, \Delta_I + G_I \, \Delta_R = 0 .\]
Therefore,
 \[    G^{-1} =  \left[\begin{array}{cc} \Delta_R & -\Delta_I \\
      \Delta _I & \Delta_R \end{array} \right]  
%      = I- Q
      ,\]
and $J^2 = G^{-1}$ where
 \[     J = \left[\begin{array}{cc} K_R & -K_I \\
      K _I & K_R \end{array} \right] . \]
In particular, $|J\phi|^2 = |K\psi|^2$. 
   Moreover if $\lambda > 0$
 is an eigenvalue of $G'$ with
 eigenvector $ \x + i \y$, then
 \[        G_R \, \x - G_I \y = \lambda \, \x,\;\;\;\;
    G_R \, \y + G_I \, \x = \lambda \, \y,\]
   from which we see that
   \[   G \, \left[\begin{array}{c} \x \\
   \y \end{array}
    \right] = \lambda \, \left[\begin{array}{c} \x \\
   \y \end{array}\right],\;\;\;\;
     G\, \left[\begin{array}{c} -\y \\
   \x \end{array}
    \right] = \lambda \, \left[\begin{array}{c} -\y \\
   \x \end{array}\right] .\]
Since the eigenvalues of $G$ are the eigenvalues of
$G'$ with double the multiplicity,
\[    \det G = [\det G']^2.\]
Therefore, \eqref{cdensity} can be written as
\[  \frac{1}{(\sqrt{2\pi})^{2n } \,  \sqrt{\det G}}\, 
             \exp\left(-\frac12  \left|J\phi\right|^2
               \right),\]
which is the density for the centered real field on $\overline A$
with covariance matrix $G$.
\end{proof}

We will discuss the analogue to theorem \ref{gffLoop}
for complex Hermitian weights.  We can either
use the complex
weights $ Q' = I - (G')^{-1} = Q_R + i \, Q_I $ 
on $A$ to give a representation of $\{|\psi(x)|^2:
 x \in A\}$ or we can use the   weights on $\overline A$ given by
\begin{equation}  \label{feb26.1}
   Q  = \left[
\begin{array}{cc}   Q_R & -Q_I \\
 Q_I & Q_R \end{array} \right]= I -G^{-1}, 
 \end{equation}
 to give a representation of $\{|\phi(z)|^2: z \in \overline A\}$.
 The latter contains more information so we will do this.
 Note that $Q$
  is a positive definite
symmetric matrix, but may not be acceptable even if $Q'$ is.
 
  Provided that $Q'$ and $Q$ are acceptable, let $ \cm , m$ denote the loop measures
derived from them respectively. As before,
let $\loopset$ denote the set of (rooted)
loops of positive length in $A$.  Let $\overline \loopset$ be the
set of such  loops in $\overline A$. Note that $ \cm$
is a complex measure on $\loopset$ and $m$ is a
real  measure
on $\overline \loopset$.  We write
\[    \cm (\omega) =  \cm_R(\omega) + i \,   \cm_I(\omega).\]
Recall that $\Phi:\overline A \rightarrow A$ is the
covering map.  We also write 
  $\Phi: \overline \loopset \rightarrow \loopset$
for the projection, that is, 
 if $\omega'=[\omega_0',
\ldots,\omega_k']\in \overline \loopset$ then 
$\Phi(\omega') $ 
is the loop of length $k$  whose $j$th component is
$  \Phi(\omega_j')$.  We define the pushforward
measure $\Phi_* m$ on $\loopset$ by
\[   \Phi_* m(\omega) = m\left[\Phi^{-1}(\omega)\right]
 = \sum_{\Phi(\omega') = \omega}
     m(\omega'). \]
     
 \begin{proposition} \label{straw}
 \[     \Phi_* m(\omega)  = 2 \, \cm_R(\omega)
   = \cm (\omega) +  \cm (\omega^R). \]
  \end{proposition}

\begin{proof}
Let $\sequences_k = \{R,I\}^k$ and if $\pi =(\pi^1,\ldots,
\pi^k) \in \sequences_k$ we write $d(\pi)$ for the number
of components that equal $I$. Let $\sequences_k^e$ denote
the set of sequences $\pi \in \sequences_k$
with $d(\pi)$ even.

 Suppose $\omega  =[\omega_0,\ldots,\omega_k]
  \in\loopset$. There
are $2^k$ loops $\omega' =[\omega_0',\omega_1',
\ldots, \omega_k'] \in \overline \loopset$    such that  
$\Phi(\omega') = \omega$.
We can write each  such loop   as an ordered triple
$(\omega, \theta, \pi)$.  Here  
$\theta \in \{0,*\}$ and $\pi \in \sequences_k^e$. 
  We obtain $\omega'$ from $(\omega, \theta, \pi)$ as follows.
If $\theta = 0$  then $\omega_0' = \omega_0$, and otherwise
$\omega_0' = \omega_0^*$.  For $j \geq 1$, $\omega_j' 
\in\{\omega_j, \omega_j^*\}$.  If $\pi^j = R$, then
 $\omega_j'$  is chosen
 to be in the same sheet as $\omega_{j-1}'$.
If $\pi^j =  I$, then  $\omega_j'$
is chosen  in the opposite
sheet to $\omega_{j-1}'$.  Since $d(\pi)$ is even,
we see that $\omega_n' = \omega_0'$ so this is gives
a loop in $\overline \loopset$ with
  $\Phi (\omega') = \omega$.

By expanding the product we see that 
\[   Q'(\omega) = \prod_{j=1}^k 
  \left[Q_R(\omega_{j-1},\omega_j) + i \, Q_I(\omega_{j-1}
  , \omega_j ) \right]
   = \sum_{\pi \in \sequences_k} i^{d(\pi)}
        \prod_{j=1}^k 
         \, Q_{\pi^j} (\omega_{j-1},\omega_j) , \]
  \[   \Re \left[  Q'(\omega) \right]
    = \sum_{\pi \in \sequences_k^e} i^{d(\pi)}
        \prod_{j=1}^k 
         \, Q_{\pi^j} (\omega_{j-1},\omega_j) , \]
 Note that
\[    Q(\omega_{j-1}',\omega_j') = - Q_I(\omega_{j-1},\omega_j),
\;\;\;\; \omega_{j-1}' \in A , \;\;\; \omega_j' \in A' , \]
 \[ Q(\omega_{j-1}',\omega_j') = Q_I(\omega_{j-1},\omega_j),
\;\;\;\; \omega_{j-1}' \in A' , \;\;\; \omega_j' \in A ,\]
\[   Q(\omega_{j-1}',\omega_j') = Q_R(\omega_{j-1},\omega_j),
\;\;\;\;\mbox{ otherwise }. \]
If $d[\pi]$ is even, then $d[\pi]/2$ denotes the number
of times that the path $\omega'$ goes from $A^*$ to $A$. 
Using this we can write
\[    \Re \left[  Q'(\omega) \right]
=\frac 12 \,Q\left[\Phi^{-1}(\omega)\right] 
    =  \frac 12  \sum_{\Phi(\omega')
     = \omega}  
        Q(\omega') . \]
  The factor $1/2$ compensates
 for the initial choice of $\omega_0'$.  Since
$  Q' (\omega^R)  = \overline{  Q'(\omega)}$, we
see that
\[     Q'(\omega) +   Q'(\omega^R)
 = Q\left[\Phi^{-1}(\omega)\right] .\]
Since
\[      \cm(\omega) = \frac{  Q'(\omega)}{|\omega|}, \;\;
\;\; 
   \Phi_* \mu(\omega) = \frac{Q\left[\Phi^{-1}(\omega)\right]}{|\omega|} ,  \]
        we get the result.
         \end{proof}
Since 
\[          \det G' = \exp \left\{   \sum_{\omega \subset A}
    \cm (\omega)\right\}, \;\;\;\;
       \det G =     \exp \left\{   \sum_{\omega' \subset 
       \overline A}
   m(\omega)\right\}, \]
we get another derivation of the relation
$             \det G= [\det G']^2.$

Given the loop measure
$  \cm$ on $A$  (or  the loop
measure $m$ on $\overline A$), we can consider the discrete
occupation field at time $t$ as a measure $ \mu_{t,\cm}$
  on $\N^A$ (or $\mu_{t,m}$ on $\N^{\overline A}$,
  respectively).  The measure $\mu_{t,m}$ pushes forward to a
  measure $\Phi_* \mu_{t,m}$ on $\N^A$ by adding the
  components of $x$ and $x^*$.
 It follows from \eqref{feb28.8} and proposition
 \ref{straw} that 
   \[     \Phi_* \mu_{t,m} =   \mu_{2t, \cm} . \]
 Also  the
 ``trivial loop occupation field'' on $\overline A$ 
 at time $t$ induces an occupation field on $A$
 by addition.  This has the same distribution
 as the trivial loop occupation
 field on $A$ at time $2t$ since there are two points in
 $\overline A$ corresponding to each point in $A$.
 Hence $\Phi_* \rho_{t,m}$ has the same distribution as
 $\rho_{2t,\cm}$.

 Using theorem \ref{gffLoop} and proposition \ref{water}
 we get the following.
 
 \begin{itemize}
 \item Suppose $ Q'$ is a positive definite acceptable
 Hermitian 
 matrix indexed by $A$.
   Let $G' = (I-  Q')^{-1}$. 
 
 \item Let $Q$ be the positive definite real matrix
 on $\overline A$ as in \eqref{feb26.1}.  Let $G =  (I-Q)
 ^{-1} $. 
 
 \item Let $\{\phi(z): z \in \overline A\}$ be 
   a  centered
  Gaussian free field on $\overline A$ with covariance
 matrix $G$ provided $Q$ is acceptable.
 
 \item  If $h(x) = [\phi(x) + i \, \phi(x^*)]/\sqrt 2$, then
 $h$ is a complex Gaussian free field on $A$ with
 covariance matrix $ G'$. 
 
 \item   If $\rho_{t}$ denotes
 the continuous occupation field  on $\overline A$
 (including trivial loops)
 given by $Q$ at time $t$ 
 then  $\{\frac 12 \phi(z)^2:
  z \in \overline A\}$ has distribution
  $\rho_{\frac 12} .$

 \item  If $\rho_t'$ denotes
 the continuous occupation field  on  $ A$
 (including trivial loops)
 given by $Q'$ at time $t$ 
 then  $\{  |h(z)|^2:
  z \in  A\}$ has distribution
  $\rho_1' .$
 
 \end{itemize}
 
 \section{Acknowledgment}
 
  The authors would like to thank Dapeng Zhan for bringing an error in an earlier version of this paper to our attention.

\end{document}